%% file: basket_problem_nature.tex
\documentclass[10pt,twocolumn]{article}

\usepackage[margin=2cm,columnsep=0.6cm]{geometry}
\usepackage{amsmath,amssymb,amsthm}
\usepackage{tikz}
\usetikzlibrary{calc,patterns,decorations.pathreplacing,arrows.meta,positioning}
\usepackage{booktabs}
\usepackage{xcolor}
\usepackage{microtype}
\usepackage{longtable}
\usepackage{colortbl}
\usepackage{array}
\usepackage{pgfplots}
\pgfplotsset{compat=1.18}
\usepackage{helvet}
\usepackage{mathpazo}
\usepackage[colorlinks=true,linkcolor=blue!60!black,citecolor=blue!60!black,urlcolor=blue!60!black]{hyperref}

\usepackage{titlesec}
\titleformat{\section}{\large\bfseries\sffamily}{\thesection}{0.5em}{}
\titleformat{\subsection}{\normalsize\bfseries\sffamily}{\thesubsection}{0.5em}{}
\titlespacing{\section}{0pt}{1.2em}{0.4em}
\titlespacing{\subsection}{0pt}{0.8em}{0.3em}

\newtheorem{theorem}{Theorem}
\newtheorem{lemma}[theorem]{Lemma}
\newtheorem{proposition}[theorem]{Proposition}

\theoremstyle{definition}
\newtheorem{definition}[theorem]{Definition}

\theoremstyle{remark}
\newtheorem*{remark}{Remark}

\definecolor{applegreen}{RGB}{106,168,79}
\definecolor{pearamber}{RGB}{212,160,23}
\definecolor{basketbrown}{RGB}{140,100,60}
\definecolor{proofblue}{RGB}{44,82,130}
\definecolor{accentred}{RGB}{180,50,50}

\newcommand{\drawapple}[3]{%
  \fill[red!70!black] (#1,#2) circle ({0.1*#3});
  \pgfmathsetmacro{\stemtop}{#2+0.15*#3}
  \pgfmathsetmacro{\stembot}{#2+0.1*#3}
  \pgfmathsetmacro{\stemx}{#1+0.01*#3}
  \pgfmathsetmacro{\stemlw}{0.3*#3}
  \draw[brown!50!black, line width=\stemlw pt] (#1,\stembot) -- (\stemx,\stemtop);
}

\newcommand{\drawpear}[3]{%
  \pgfmathsetmacro{\pA}{#1-0.08*#3}
  \pgfmathsetmacro{\pB}{#2+0.04*#3}
  \pgfmathsetmacro{\pC}{#1-0.09*#3}
  \pgfmathsetmacro{\pD}{#2+0.12*#3}
  \pgfmathsetmacro{\pE}{#1-0.03*#3}
  \pgfmathsetmacro{\pF}{#2+0.18*#3}
  \pgfmathsetmacro{\pG}{#2+0.22*#3}
  \pgfmathsetmacro{\pH}{#1+0.03*#3}
  \pgfmathsetmacro{\pI}{#1+0.09*#3}
  \pgfmathsetmacro{\pJ}{#1+0.08*#3}
  \fill[yellow!55!green!70] (#1,#2)
    .. controls (\pA,\pB) and (\pC,\pD) .. (\pE,\pF)
    .. controls (#1,\pG) and (#1,\pG) .. (\pH,\pF)
    .. controls (\pI,\pD) and (\pJ,\pB) .. cycle;
}

\title{\textsf{\textbf{The Apple\,--\,Pear Basket Problem:\\A Combinatorial Exploration}}}
\author{Rethna Pulikkoonattu}
\date{}

\begin{document}
\twocolumn[
  \maketitle
  \begin{center}
  \begin{minipage}{0.85\textwidth}
  \small
  \textbf{Abstract.}
  We investigate a combinatorial puzzle in which $N$~apples and $N$~pears are distributed among baskets subject to two constraints: every basket must contain the same number of apples, and every basket must contain a distinct number of pears.  We prove that the maximum number of baskets is the largest divisor of~$N$ not exceeding $(1 + \sqrt{1+8N})/2$.  For the original puzzle with $N = 60$, this yields $10$~baskets.  The solution reveals a rich interplay between divisibility and combinatorics, leading to a natural classification of integers into perfect values, primes, and highly composite numbers according to their basket-packing efficiency.  Computational results for $N$ up to one million confirm the asymptotic growth rate of $\sqrt{2N}$, and a complete tabulation for $N = 1$ to $100$ is included.

  \medskip
  \end{minipage}
  \end{center}
  \vspace{0.5em}
]

\section{Introduction}

Consider the following puzzle, which the author encountered online.  You have $N$~apples and $N$~pears.  You wish to divide all the fruit into baskets subject to two rules: every basket must contain the same number of apples, and every basket must contain a different number of pears.  What is the maximum number of baskets you can use?  Problems of this flavour appear frequently in recreational mathematics~\cite{conway1996,gardner1997}, but this particular formulation, to the author's knowledge, has not been analysed in the literature.

\begin{remark}
The original puzzle posed $N = 60$.  All generalisations, proofs, and the classification framework presented here are the author's own work.
\end{remark}

The problem is easy to state but surprisingly rich.  As we shall see, its solution depends on an interplay between two branches of elementary mathematics: the divisibility properties of~$N$, a number-theoretic condition rooted in classical results on the divisor function~\cite{hardy2008}, and the minimum sum of distinct non-negative integers, a combinatorial condition related to the theory of partitions~\cite{andrews1998}.

\section{The Two Constraints}

Suppose we use $n$~baskets, each containing $k$~apples.  Then $nk = N$, so $n$ must divide $N$.

For the pears, let $p_1, p_2, \ldots, p_n$ denote the (distinct, non-negative) number of pears in each basket, with $p_1 + p_2 + \cdots + p_n = N$.  We seek the minimum possible value of this sum.

\begin{lemma}[Minimum-Sum Lemma]
\label{lem:minsum}
The minimum sum of $n$ distinct non-negative integers is
\[
  0 + 1 + 2 + \cdots + (n-1) = \frac{n(n-1)}{2}\,,
\]
achieved uniquely by the set $\{0, 1, 2, \ldots, n-1\}$.
\end{lemma}

\begin{proof}
Let $a_1 < a_2 < \cdots < a_n$ be $n$~distinct non-negative integers.  Since they are distinct integers with $a_1 \ge 0$, we have $a_i \ge i - 1$ for each $i$.  Therefore
\[
  \sum_{i=1}^n a_i \ge \sum_{i=1}^n (i-1) = \frac{n(n-1)}{2}\,,
\]
with equality if and only if $a_i = i - 1$ for all $i$.
\end{proof}

Since the pear total must equal exactly~$N$, we need $n(n-1)/2 \le N$.  The quantity $n(n-1)/2$ is the $n$th triangular number~\cite{oeis_triangular,hardy2008}, a connection that will recur throughout this paper.  Solving for $n$:
\[
  n \;\le\; \frac{1 + \sqrt{1 + 8N}}{2}\,.
\]

\section{The General Solution}

Combining both constraints yields the main result.

\begin{theorem}[General Solution]
\label{thm:general}
The maximum number of baskets is
\[
  n_{\max} = \max\!\left\{\, d \in \mathbb{Z}^+ : d \mid N \;\text{and}\; d \le \frac{1 + \sqrt{1+8N}}{2} \right\}.
\]
\end{theorem}

\begin{proof}
Any valid $n$ must satisfy both $n \mid N$ (apple constraint) and $n(n-1)/2 \le N$ (pear constraint).  Among all such divisors, the largest one gives $n_{\max}$.
\end{proof}

\begin{remark}
A valid pear distribution always exists when both constraints hold.  Given any $n$ with $n \mid N$ and $n(n-1)/2 \le N$, assign pears as $p_i = i - 1$ for $i = 1, \ldots, n-1$ and $p_n = N - n(n-1)/2 + (n-1)$.  Since $N - n(n-1)/2 \ge 0$, we have $p_n \ge n - 1 > p_{n-1}$, so all values remain distinct.
\end{remark}

For the original puzzle with $N = 60$: the pear bound is $(1 + \sqrt{481})/2 \approx 11.47$.  The divisors of $60$ that do not exceed this bound are $\{1, 2, 3, 4, 5, 6, 10\}$.  The largest is $n_{\max} = 10$, with $k = 6$ apples per basket and pears distributed as $\{0, 1, 2, 3, 4, 5, 6, 7, 8, 24\}$.  The complete solution is illustrated in Figure~\ref{fig:closing}.

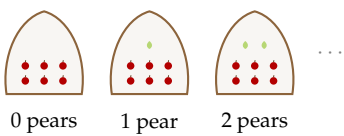
\begin{figure}[t]
\centering
\begin{tikzpicture}[scale=0.5]
  \foreach \b/\p/\xoff in {1/0/0, 2/1/2.8, 3/2/5.6} {
    \draw[thick, basketbrown, fill=basketbrown!6]
      (\xoff-1.0,-0.2) .. controls (\xoff-1.15,0.8) and (\xoff-0.7,1.8) ..
      (\xoff,2.0) .. controls (\xoff+0.7,1.8) and (\xoff+1.15,0.8) ..
      (\xoff+1.0,-0.2) -- cycle;
    \drawapple{\xoff-0.5}{0.15}{1}
    \drawapple{\xoff}{0.15}{1}
    \drawapple{\xoff+0.5}{0.15}{1}
    \drawapple{\xoff-0.5}{0.55}{1}
    \drawapple{\xoff}{0.55}{1}
    \drawapple{\xoff+0.5}{0.55}{1}
  }
  \drawpear{2.8}{1.0}{1}
  \drawpear{5.35}{1.0}{1}
  \drawpear{5.85}{1.0}{1}
  \node[below, font=\footnotesize] at (0,-0.45) {0 pears};
  \node[below, font=\footnotesize] at (2.8,-0.45) {1 pear};
  \node[below, font=\footnotesize] at (5.6,-0.45) {2 pears};
  \node[right, font=\footnotesize\itshape, text=gray] at (7.0,0.9) {$\cdots$};
\end{tikzpicture}
\caption{Each basket has the same number of apples (red) but a distinct number of pears (green).}
\label{fig:baskets}
\end{figure}

\section{Classification of $N$}

The interplay between the smooth pear bound and the arithmetic divisibility constraint creates a rich landscape.

\subsection{Perfect Values}
\label{sec:perfect}

\begin{definition}
We call $N$ \emph{perfect} for this problem if the pear constraint is exactly tight and $n \mid N$: that is, $N = n(n-1)/2$ and $n \mid n(n-1)/2$.
\end{definition}

Since $n(n-1)/2$ divided by $n$ equals $(n-1)/2$, divisibility holds if and only if $n$ is odd.  This gives the family shown in Figure~\ref{fig:triangular}, whose members are precisely the even-indexed triangular numbers~\cite{oeis_triangular}.

\begin{figure*}[!t]
\centering
\begin{tikzpicture}[scale=0.55]

  \foreach \row in {0,...,3} {
    \pgfmathsetmacro{\ypos}{(3-\row)*0.24}
    \foreach \col in {0,...,\row} {
      \fill[pearamber] (\col*0.26 - \row*0.13, \ypos) circle (0.10);
    }
  }
  \node[font=\small, below] at (0, -0.3) {$T_4 = 10$};

  \begin{scope}[xshift=2.5cm]
  \foreach \row in {0,...,5} {
    \pgfmathsetmacro{\ypos}{(5-\row)*0.24}
    \foreach \col in {0,...,\row} {
      \fill[proofblue, opacity=0.7] (\col*0.26 - \row*0.13, \ypos) circle (0.10);
    }
  }
  \node[font=\small, below] at (0, -0.3) {$T_6 = 21$};
  \end{scope}

  \begin{scope}[xshift=5.5cm]
  \foreach \row in {0,...,7} {
    \pgfmathsetmacro{\ypos}{(7-\row)*0.24}
    \foreach \col in {0,...,\row} {
      \fill[applegreen, opacity=0.7] (\col*0.26 - \row*0.13, \ypos) circle (0.10);
    }
  }
  \node[font=\small, below] at (0, -0.3) {$T_8 = 36$};
  \end{scope}

  \begin{scope}[xshift=9.5cm]
  \foreach \row in {0,...,9} {
    \pgfmathsetmacro{\ypos}{(9-\row)*0.24}
    \foreach \col in {0,...,\row} {
      \fill[accentred, opacity=0.6] (\col*0.26 - \row*0.13, \ypos) circle (0.10);
    }
  }
  \node[font=\small, below] at (0, -0.3) {$T_{10} = 55$};
  \end{scope}

  \begin{scope}[xshift=14cm]
  \foreach \row in {0,...,11} {
    \pgfmathsetmacro{\ypos}{(11-\row)*0.24}
    \foreach \col in {0,...,\row} {
      \fill[basketbrown, opacity=0.6] (\col*0.26 - \row*0.13, \ypos) circle (0.10);
    }
  }
  \node[font=\small, below] at (0, -0.3) {$T_{12} = 78$};
  \end{scope}

  \begin{scope}[xshift=19.5cm]
  \foreach \row in {0,...,13} {
    \pgfmathsetmacro{\ypos}{(13-\row)*0.24}
    \foreach \col in {0,...,\row} {
      \fill[purple!60!blue, opacity=0.5] (\col*0.26 - \row*0.13, \ypos) circle (0.10);
    }
  }
  \node[font=\small, below] at (0, -0.3) {$T_{14} = 105$};
  \end{scope}

  \begin{scope}[xshift=25.5cm]
  \foreach \row in {0,...,15} {
    \pgfmathsetmacro{\ypos}{(15-\row)*0.24}
    \foreach \col in {0,...,\row} {
      \fill[gray!60, opacity=0.6] (\col*0.26 - \row*0.13, \ypos) circle (0.10);
    }
  }
  \node[font=\small, below] at (0, -0.3) {$T_{16} = 136$};
  \end{scope}
\end{tikzpicture}
\caption{Triangular numbers yielding perfect solutions: $T_4 = 10$ through $T_{16} = 136$.  Dot size is uniform across all seven triangles; only the triangle itself grows.  Bottom alignment emphasises the quadratic growth of the base width, corresponding to the $n^2/2$ scaling of the perfect values.}
\label{fig:triangular}
\end{figure*}
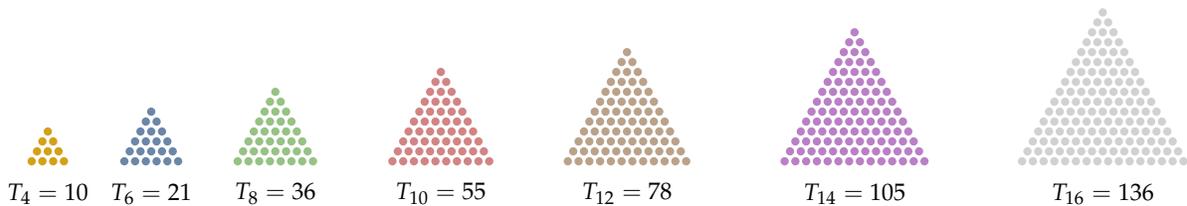

\subsection{The Prime Barrier}
\label{sec:primes}

When $N = p$ is prime, its only divisors are $1$ and $p$.  For any prime $p \ge 5$, the value of~$p$ exceeds the pear bound, so $n_{\max} = 1$.  This sharp collapse at primes is a recurring theme in problems that depend on divisor structure~\cite{hardy2008,apostol1976}.

\subsection{Highly Composite Numbers}

Numbers with many small divisors, such as $60 = 2^2 \cdot 3 \cdot 5$ or $120 = 2^3 \cdot 3 \cdot 5$, tend to have a divisor close to the pear bound.  The efficiency $n_{\max}/\text{bound}$ for these numbers is consistently high.  Ramanujan~\cite{ramanujan1915} initiated the systematic study of such highly composite numbers, and the present problem provides a natural context in which their special status becomes visible.

\section{Asymptotic Behaviour}

The pear bound grows as $\sqrt{2N}$, so $n_{\max}$ can be at most $O(\sqrt{N})$.  Perfect values achieve this rate exactly.  By contrast, primes are locked at $n_{\max} = 1$ regardless of their size.  The prime number theorem~\cite{hadamard1896,vallee1896}, which states that the number of primes up to~$N$ is asymptotic to $N/\ln N$, ensures that primes thin at rate $1/\ln N$, so the floor at $n_{\max} = 1$ becomes sparser but never vanishes.  A comprehensive treatment of the prime counting function and its asymptotics can be found in Apostol~\cite{apostol1976}.

\section{Computational Results}

To visualise the behaviour of $n_{\max}$ across a wide range of~$N$, we computed the solution of Theorem~\ref{thm:general} for every integer up to one million.  Three representative scatter plots are shown in Figures~\ref{fig:scatter200},~\ref{fig:scatter10k}, and~\ref{fig:scatter1m}, each revealing progressively more of the problem's asymptotic character.

At the smallest scale (Figure~\ref{fig:scatter200}, $N \le 200$), the fan structure is already visible: perfect values (Section~\ref{sec:perfect}) trace the pear bound from above, primes (Section~\ref{sec:primes}) are pinned to the floor at $n_{\max} = 1$, and all other integers scatter between these extremes according to their divisor structure.

\begin{figure*}[t]
\centering
\begin{tikzpicture}
  \begin{axis}[
    width=\linewidth,
    height=6.5cm,
    xlabel={$N$},
    ylabel={$n_{\max}$},
    xmin=0, xmax=200,
    ymin=0, ymax=22,
    grid=both,
    grid style={gray!10},
    tick label style={font=\scriptsize},
    label style={font=\small},
    legend style={font=\scriptsize, at={(0.02,0.98)}, anchor=north west},
  ]
  \addplot[proofblue, thick, dashed, domain=1:200, samples=100]
    {(1+sqrt(1+8*x))/2};
  \addlegendentry{Pear bound}
  \addplot[only marks, mark=*, mark size=1.2pt, applegreen!80!black] coordinates {
    (1,1)(2,2)(3,3)(4,2)(5,1)(6,3)(7,1)(8,4)(9,3)(10,5)
    (11,1)(12,4)(13,1)(14,2)(15,5)(16,4)(17,1)(18,6)(19,1)(20,5)
    (21,7)(22,2)(23,1)(24,6)(25,5)(26,2)(27,3)(28,7)(29,1)(30,6)
    (31,1)(32,8)(33,3)(34,2)(35,7)(36,9)(37,1)(38,2)(39,3)(40,8)
    (41,1)(42,7)(43,1)(44,4)(45,9)(46,2)(47,1)(48,8)(49,7)(50,10)
    (51,3)(52,4)(53,1)(54,9)(55,11)(56,8)(57,3)(58,2)(59,1)(60,10)
    (61,1)(62,2)(63,9)(64,8)(65,5)(66,12)(67,1)(68,4)(69,3)(70,10)
    (71,1)(72,12)(73,1)(74,2)(75,5)(76,4)(77,7)(78,13)(79,1)(80,10)
    (81,9)(82,2)(83,1)(84,12)(85,5)(86,2)(87,3)(88,8)(89,1)(90,9)
    (91,13)(92,4)(93,3)(94,2)(95,5)(96,12)(97,1)(98,14)(99,11)(100,10)
    (101,1)(102,6)(103,1)(104,8)(105,15)(106,2)(107,1)(108,12)(109,1)(110,10)
    (111,3)(112,8)(113,1)(114,6)(115,5)(116,4)(117,9)(118,2)(119,7)(120,15)
    (121,11)(122,2)(123,3)(124,4)(125,5)(126,9)(127,1)(128,16)(129,3)(130,10)
    (131,1)(132,12)(133,7)(134,2)(135,15)(136,17)(137,1)(138,6)(139,1)(140,8)
    (141,3)(142,2)(143,11)(144,16)(145,5)(146,2)(147,7)(148,4)(149,1)(150,10)
    (151,1)(152,8)(153,17)(154,14)(155,5)(156,12)(157,1)(158,2)(159,3)(160,16)
    (161,7)(162,18)(163,1)(164,4)(165,15)(166,2)(167,1)(168,8)(169,13)(170,10)
    (171,19)(172,4)(173,1)(174,6)(175,7)(176,16)(177,3)(178,2)(179,1)(180,18)
    (181,1)(182,14)(183,3)(184,8)(185,5)(186,6)(187,11)(188,4)(189,7)(190,10)
    (191,1)(192,16)(193,1)(194,2)(195,15)(196,14)(197,1)(198,9)(199,1)(200,20)
  };
  \addlegendentry{$n_{\max}$}
  \addplot[only marks, mark=o, mark size=2.5pt, thick, pearamber!90!black] coordinates {
    (3,3)(10,5)(21,7)(36,9)(55,11)(78,13)(105,15)(136,17)(171,19)
  };
  \addlegendentry{Perfect values}
  \addplot[only marks, mark=*, mark size=1.2pt, accentred] coordinates {
    (5,1)(7,1)(11,1)(13,1)(17,1)(19,1)(23,1)(29,1)(31,1)(37,1)(41,1)(43,1)(47,1)(53,1)(59,1)(61,1)(67,1)(71,1)(73,1)(79,1)(83,1)(89,1)(97,1)(101,1)(103,1)(107,1)(109,1)(113,1)(127,1)(131,1)(137,1)(139,1)(149,1)(151,1)(157,1)(163,1)(167,1)(173,1)(179,1)(181,1)(191,1)(193,1)(197,1)(199,1)
  };
  \addlegendentry{Primes}
  \end{axis}
\end{tikzpicture}
\caption{$n_{\max}$ for $N = 1$ to $200$.  Perfect values (amber) trace the pear bound.  Primes (red) are pinned at $n_{\max} = 1$.}
\label{fig:scatter200}
\end{figure*}
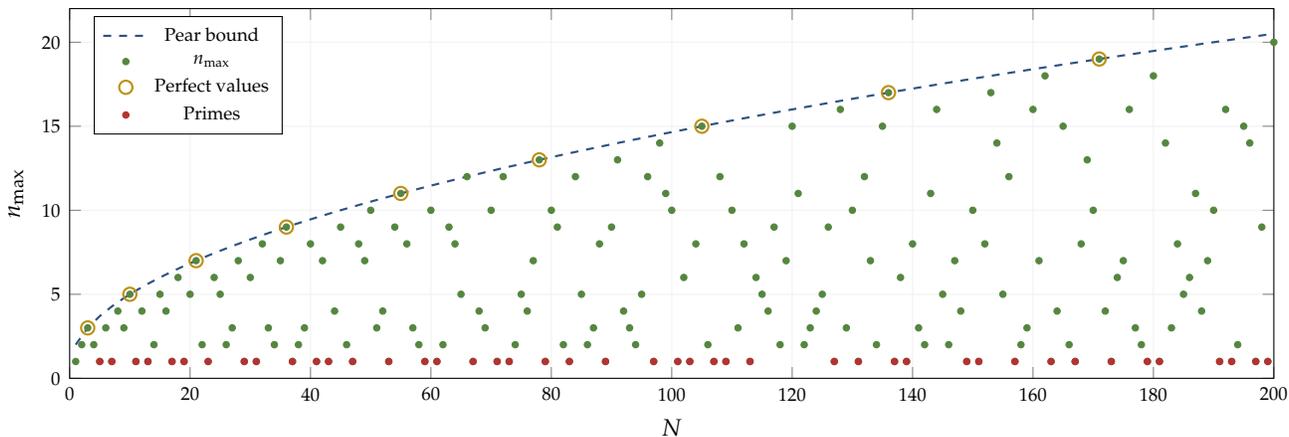

Extending the computation to larger scales reveals the full asymptotic character of the problem.  At $N = 10{,}000$ (Figure~\ref{fig:scatter10k}), the $\sqrt{2N}$ envelope becomes unmistakable, and the prime floor begins to thin visibly.  At $N = 1{,}000{,}000$ (Figure~\ref{fig:scatter1m}), the $706$~perfect values march along the envelope in perfect order, while the prime floor, though still present, has thinned to the density predicted by the prime number theorem~\cite{apostol1976}.

\begin{figure*}[t]
\centering
\begin{tikzpicture}
  \begin{axis}[
    width=\linewidth,
    height=6.5cm,
    xlabel={$N$},
    ylabel={$n_{\max}$},
    xmin=0, xmax=10000,
    ymin=0, ymax=150,
    grid=both,
    grid style={gray!10},
    tick label style={font=\scriptsize},
    label style={font=\small},
    legend style={font=\scriptsize, at={(0.02,0.98)}, anchor=north west},
  ]
  \addplot[proofblue, thick, dashed, domain=1:10000, samples=100]
    {(1+sqrt(1+8*x))/2};
  \addlegendentry{Pear bound}
  \addplot[only marks, mark=*, mark size=0.5pt, applegreen!80!black]
    table[x=N, y=nmax] {nmax_sampled.csv};
  \addlegendentry{$n_{\max}$}
  \addplot[only marks, mark=o, mark size=2pt, thick, pearamber!90!black]
    table[x=N, y=nmax] {nmax_perfect.csv};
  \addlegendentry{Perfect values}
  \addplot[only marks, mark=*, mark size=0.5pt, accentred]
    table[x=N, y=nmax] {nmax_primes_10k.csv};
  \addlegendentry{Primes}
  \end{axis}
\end{tikzpicture}
\caption{$n_{\max}$ for $N$ up to $10{,}000$.  The $\sqrt{2N}$ envelope is unmistakable.}
\label{fig:scatter10k}
\end{figure*}
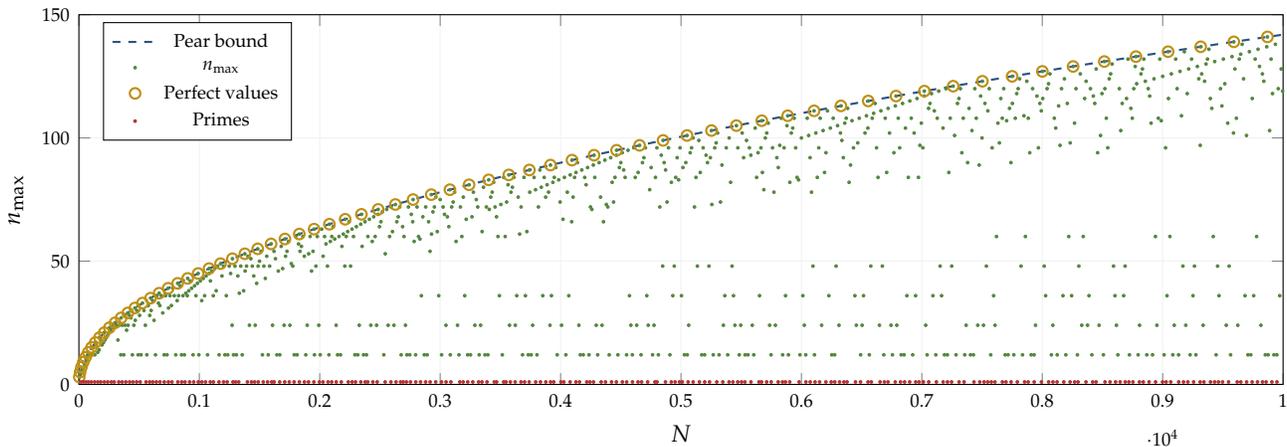

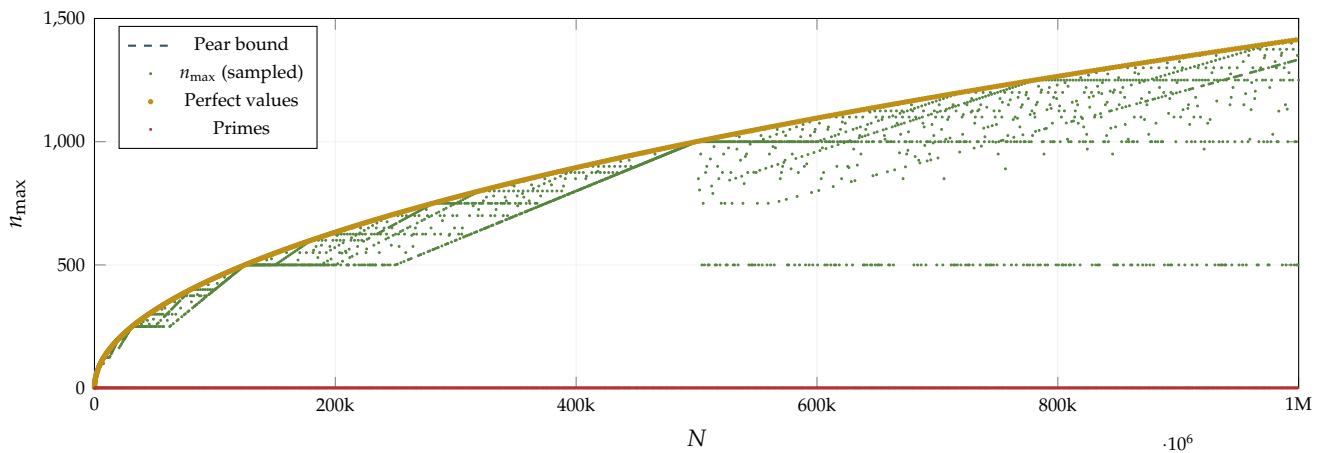
\begin{figure*}[t]
\centering
\begin{tikzpicture}
  \begin{axis}[
    width=\linewidth,
    height=6.5cm,
    xlabel={$N$},
    ylabel={$n_{\max}$},
    xmin=0, xmax=1000000,
    ymin=0, ymax=1500,
    grid=both,
    grid style={gray!10},
    tick label style={font=\scriptsize},
    label style={font=\small},
    legend style={font=\scriptsize, at={(0.02,0.98)}, anchor=north west},
    xtick={0,200000,400000,600000,800000,1000000},
    xticklabels={$0$,$200\text{k}$,$400\text{k}$,$600\text{k}$,$800\text{k}$,$1\text{M}$},
  ]
  \addplot[proofblue, thick, dashed, domain=1:1000000, samples=200]
    {(1+sqrt(1+8*x))/2};
  \addlegendentry{Pear bound}
  \addplot[only marks, mark=*, mark size=0.35pt, applegreen!80!black]
    table[x=N, y=nmax] {nmax_1m_sampled.csv};
  \addlegendentry{$n_{\max}$ (sampled)}
  \addplot[only marks, mark=*, mark size=0.8pt, pearamber!90!black]
    table[x=N, y=nmax] {nmax_1m_perfect.csv};
  \addlegendentry{Perfect values}
  \addplot[only marks, mark=*, mark size=0.35pt, accentred]
    table[x=N, y=nmax] {nmax_primes_1m.csv};
  \addlegendentry{Primes}
  \end{axis}
\end{tikzpicture}
\caption{Panoramic view: $n_{\max}$ for $N$ up to one million.  At this scale, the $706$~perfect values trace the $\sqrt{2N}$ envelope while the prime floor thins according to the prime number theorem.}
\label{fig:scatter1m}
\end{figure*}

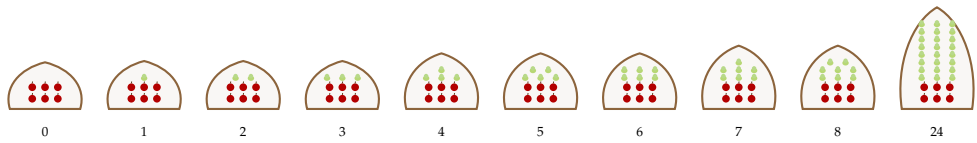
\begin{figure*}[t]
\centering
\begin{tikzpicture}[scale=0.85]
  \input{closing_baskets.tex}
\end{tikzpicture}
\caption{The complete solution for $N = 60$: ten baskets, each containing $6$~apples (red), with pears (green) distributed as $\{0, 1, 2, 3, 4, 5, 6, 7, 8, 24\}$.  Every apple and pear is shown individually.  The last basket's $24$~pears visually demonstrate the surplus concentration described in Section~\ref{sec:surplus}.}
\label{fig:closing}
\end{figure*}

\section{The Surplus Distribution}
\label{sec:surplus}

When $n_{\max}$ is determined by Theorem~\ref{thm:general} and the minimum pear sum $n(n-1)/2$ is strictly less than~$N$, there is a surplus of $S = N - n(n-1)/2$ pears that must be allocated among the baskets while preserving the distinctness of the pear counts.  The simplest valid strategy, used throughout this paper, is to add the entire surplus to the last basket, giving the distribution $\{0, 1, \ldots, n-2, \, n-1+S\}$.  However, any redistribution that maintains distinct non-negative values summing to~$N$ is equally valid.

\begin{proposition}
For any $n$ and $N$ satisfying the constraints of Theorem~\ref{thm:general}, the number of valid pear distributions is at least~$1$ and grows with the surplus $S = N - n(n-1)/2$.
\end{proposition}

For the original puzzle with $N = 60$ and $n = 10$, the surplus is $S = 60 - 45 = 15$, and the canonical distribution $\{0, 1, 2, 3, 4, 5, 6, 7, 8, 24\}$ places all $15$ surplus pears in the last basket, as illustrated in Figure~\ref{fig:closing}.

\section{Tabulated Results}

Table~\ref{tab:full} lists the solution of Theorem~\ref{thm:general} for every integer $N$ from~$1$ to~$100$, showing the pear bound, $n_{\max}$, apples per basket~$k$, efficiency, and the full pear distribution.  Rows are shaded amber for perfect values, red for primes, and green for near-perfect values (efficiency $> 0.9$).

\begin{table*}[t]
\centering
\caption{Complete results for $N = 1$ to $100$ by Theorem~\ref{thm:general}.  Left half: $N = 1$ to $50$.  Right half: $N = 51$ to $100$.}\label{tab:full}
\vspace{0.3em}
{\fontsize{6.5}{8.5}\selectfont
\setlength{\tabcolsep}{2.5pt}
\begin{tabular}{r r r r r l @{\quad} r r r r r l}
\toprule
$N$ & Bnd & $n$ & $k$ & Eff & Pear distribution & $N$ & Bnd & $n$ & $k$ & Eff & Pear distribution \\
\midrule
    1 & 2.0 & \textbf{1} & 1 & 0.50 & $\{1\}$ & 51 & 10.6 & \textbf{3} & 17 & 0.28 & $\{0, 1, 50\}$ \\
    {\color{accentred}2} & 2.6 & \textbf{2} & 1 & 0.78 & $\{0, 2\}$ & 52 & 10.7 & \textbf{4} & 13 & 0.37 & $\{0, 1, 2, 49\}$ \\
    {\color{pearamber!80!black}3} & 3.0 & \textbf{3} & 1 & 1.00 & $\{0, 1, 2\}$ & {\color{accentred}53} & 10.8 & \textbf{1} & 53 & 0.09 & $\{53\}$ \\
    4 & 3.4 & \textbf{2} & 2 & 0.59 & $\{0, 4\}$ & 54 & 10.9 & \textbf{9} & 6 & 0.83 & $\{0, 1, 2, 3, 4, 5, 6, 7, 26\}$ \\
    {\color{accentred}5} & 3.7 & \textbf{1} & 5 & 0.27 & $\{5\}$ & {\color{pearamber!80!black}55} & 11.0 & \textbf{11} & 5 & 1.00 & $\{0, 1, 2, 3, 4, 5, 6, 7, 8, 9, 10\}$ \\
    6 & 4.0 & \textbf{3} & 2 & 0.75 & $\{0, 1, 5\}$ & 56 & 11.1 & \textbf{8} & 7 & 0.72 & $\{0, 1, 2, 3, 4, 5, 6, 35\}$ \\
    {\color{accentred}7} & 4.3 & \textbf{1} & 7 & 0.23 & $\{7\}$ & 57 & 11.2 & \textbf{3} & 19 & 0.27 & $\{0, 1, 56\}$ \\
    8 & 4.5 & \textbf{4} & 2 & 0.88 & $\{0, 1, 2, 5\}$ & 58 & 11.3 & \textbf{2} & 29 & 0.18 & $\{0, 58\}$ \\
    9 & 4.8 & \textbf{3} & 3 & 0.63 & $\{0, 1, 8\}$ & {\color{accentred}59} & 11.4 & \textbf{1} & 59 & 0.09 & $\{59\}$ \\
    {\color{pearamber!80!black}10} & 5.0 & \textbf{5} & 2 & 1.00 & $\{0, 1, 2, 3, 4\}$ & 60 & 11.5 & \textbf{10} & 6 & 0.87 & $\{0, 1, 2, 3, 4, 5, 6, 7, 8, 24\}$ \\
    {\color{accentred}11} & 5.2 & \textbf{1} & 11 & 0.19 & $\{11\}$ & {\color{accentred}61} & 11.6 & \textbf{1} & 61 & 0.09 & $\{61\}$ \\
    12 & 5.4 & \textbf{4} & 3 & 0.74 & $\{0, 1, 2, 9\}$ & 62 & 11.6 & \textbf{2} & 31 & 0.17 & $\{0, 62\}$ \\
    {\color{accentred}13} & 5.6 & \textbf{1} & 13 & 0.18 & $\{13\}$ & 63 & 11.7 & \textbf{9} & 7 & 0.77 & $\{0, 1, 2, 3, 4, 5, 6, 7, 35\}$ \\
    14 & 5.8 & \textbf{2} & 7 & 0.34 & $\{0, 14\}$ & 64 & 11.8 & \textbf{8} & 8 & 0.68 & $\{0, 1, 2, 3, 4, 5, 6, 43\}$ \\
    15 & 6.0 & \textbf{5} & 3 & 0.83 & $\{0, 1, 2, 3, 9\}$ & 65 & 11.9 & \textbf{5} & 13 & 0.42 & $\{0, 1, 2, 3, 59\}$ \\
    16 & 6.2 & \textbf{4} & 4 & 0.65 & $\{0, 1, 2, 13\}$ & {\color{applegreen!70!black}66} & 12.0 & \textbf{11} & 6 & 0.92 & $\{0, 1, 2, 3, 4, 5, 6, 7, 8, 9, 21\}$ \\
    {\color{accentred}17} & 6.4 & \textbf{1} & 17 & 0.16 & $\{17\}$ & {\color{accentred}67} & 12.1 & \textbf{1} & 67 & 0.08 & $\{67\}$ \\
    {\color{applegreen!70!black}18} & 6.5 & \textbf{6} & 3 & 0.92 & $\{0, 1, 2, 3, 4, 8\}$ & 68 & 12.2 & \textbf{4} & 17 & 0.33 & $\{0, 1, 2, 65\}$ \\
    {\color{accentred}19} & 6.7 & \textbf{1} & 19 & 0.15 & $\{19\}$ & 69 & 12.3 & \textbf{3} & 23 & 0.24 & $\{0, 1, 68\}$ \\
    20 & 6.8 & \textbf{5} & 4 & 0.73 & $\{0, 1, 2, 3, 14\}$ & 70 & 12.3 & \textbf{10} & 7 & 0.81 & $\{0, 1, 2, 3, 4, 5, 6, 7, 8, 34\}$ \\
    {\color{pearamber!80!black}21} & 7.0 & \textbf{7} & 3 & 1.00 & $\{0, 1, 2, 3, 4, 5, 6\}$ & {\color{accentred}71} & 12.4 & \textbf{1} & 71 & 0.08 & $\{71\}$ \\
    22 & 7.2 & \textbf{2} & 11 & 0.28 & $\{0, 22\}$ & {\color{applegreen!70!black}72} & 12.5 & \textbf{12} & 6 & 0.96 & $\{0, 1, 2, 3, 4, 5, 6, 7, 8, 9, 10, 17\}$ \\
    {\color{accentred}23} & 7.3 & \textbf{1} & 23 & 0.14 & $\{23\}$ & {\color{accentred}73} & 12.6 & \textbf{1} & 73 & 0.08 & $\{73\}$ \\
    24 & 7.4 & \textbf{6} & 4 & 0.81 & $\{0, 1, 2, 3, 4, 14\}$ & 74 & 12.7 & \textbf{2} & 37 & 0.16 & $\{0, 74\}$ \\
    25 & 7.6 & \textbf{5} & 5 & 0.66 & $\{0, 1, 2, 3, 19\}$ & 75 & 12.8 & \textbf{5} & 15 & 0.39 & $\{0, 1, 2, 3, 69\}$ \\
    26 & 7.7 & \textbf{2} & 13 & 0.26 & $\{0, 26\}$ & 76 & 12.8 & \textbf{4} & 19 & 0.31 & $\{0, 1, 2, 73\}$ \\
    27 & 7.9 & \textbf{3} & 9 & 0.38 & $\{0, 1, 26\}$ & 77 & 12.9 & \textbf{11} & 7 & 0.85 & $\{0, 1, 2, 3, 4, 5, 6, 7, 8, 9, 32\}$ \\
    28 & 8.0 & \textbf{7} & 4 & 0.88 & $\{0, 1, 2, 3, 4, 5, 13\}$ & {\color{pearamber!80!black}78} & 13.0 & \textbf{13} & 6 & 1.00 & $\{0, 1, 2, 3, 4, 5, 6, 7, 8, 9, 10, 11, 12\}$ \\
    {\color{accentred}29} & 8.1 & \textbf{1} & 29 & 0.12 & $\{29\}$ & {\color{accentred}79} & 13.1 & \textbf{1} & 79 & 0.08 & $\{79\}$ \\
    30 & 8.3 & \textbf{6} & 5 & 0.73 & $\{0, 1, 2, 3, 4, 20\}$ & 80 & 13.2 & \textbf{10} & 8 & 0.76 & $\{0, 1, 2, 3, 4, 5, 6, 7, 8, 44\}$ \\
    {\color{accentred}31} & 8.4 & \textbf{1} & 31 & 0.12 & $\{31\}$ & 81 & 13.2 & \textbf{9} & 9 & 0.68 & $\{0, 1, 2, 3, 4, 5, 6, 7, 53\}$ \\
    {\color{applegreen!70!black}32} & 8.5 & \textbf{8} & 4 & 0.94 & $\{0, 1, 2, 3, 4, 5, 6, 11\}$ & 82 & 13.3 & \textbf{2} & 41 & 0.15 & $\{0, 82\}$ \\
    33 & 8.6 & \textbf{3} & 11 & 0.35 & $\{0, 1, 32\}$ & {\color{accentred}83} & 13.4 & \textbf{1} & 83 & 0.07 & $\{83\}$ \\
    34 & 8.8 & \textbf{2} & 17 & 0.23 & $\{0, 34\}$ & 84 & 13.5 & \textbf{12} & 7 & 0.89 & $\{0, 1, 2, 3, 4, 5, 6, 7, 8, 9, 10, 29\}$ \\
    35 & 8.9 & \textbf{7} & 5 & 0.79 & $\{0, 1, 2, 3, 4, 5, 20\}$ & 85 & 13.5 & \textbf{5} & 17 & 0.37 & $\{0, 1, 2, 3, 79\}$ \\
    {\color{pearamber!80!black}36} & 9.0 & \textbf{9} & 4 & 1.00 & $\{0, 1, 2, 3, 4, 5, 6, 7, 8\}$ & 86 & 13.6 & \textbf{2} & 43 & 0.15 & $\{0, 86\}$ \\
    {\color{accentred}37} & 9.1 & \textbf{1} & 37 & 0.11 & $\{37\}$ & 87 & 13.7 & \textbf{3} & 29 & 0.22 & $\{0, 1, 86\}$ \\
    38 & 9.2 & \textbf{2} & 19 & 0.22 & $\{0, 38\}$ & 88 & 13.8 & \textbf{11} & 8 & 0.80 & $\{0, 1, 2, 3, 4, 5, 6, 7, 8, 9, 43\}$ \\
    39 & 9.3 & \textbf{3} & 13 & 0.32 & $\{0, 1, 38\}$ & {\color{accentred}89} & 13.9 & \textbf{1} & 89 & 0.07 & $\{89\}$ \\
    40 & 9.5 & \textbf{8} & 5 & 0.85 & $\{0, 1, 2, 3, 4, 5, 6, 19\}$ & 90 & 13.9 & \textbf{10} & 9 & 0.72 & $\{0, 1, 2, 3, 4, 5, 6, 7, 8, 54\}$ \\
    {\color{accentred}41} & 9.6 & \textbf{1} & 41 & 0.10 & $\{41\}$ & {\color{applegreen!70!black}91} & 14.0 & \textbf{13} & 7 & 0.93 & $\{0, 1, 2, 3, 4, 5, 6, 7, 8, 9, 10, 11, 25\}$ \\
    42 & 9.7 & \textbf{7} & 6 & 0.72 & $\{0, 1, 2, 3, 4, 5, 27\}$ & 92 & 14.1 & \textbf{4} & 23 & 0.28 & $\{0, 1, 2, 89\}$ \\
    {\color{accentred}43} & 9.8 & \textbf{1} & 43 & 0.10 & $\{43\}$ & 93 & 14.1 & \textbf{3} & 31 & 0.21 & $\{0, 1, 92\}$ \\
    44 & 9.9 & \textbf{4} & 11 & 0.40 & $\{0, 1, 2, 41\}$ & 94 & 14.2 & \textbf{2} & 47 & 0.14 & $\{0, 94\}$ \\
    45 & 10.0 & \textbf{9} & 5 & 0.90 & $\{0, 1, 2, 3, 4, 5, 6, 7, 17\}$ & 95 & 14.3 & \textbf{5} & 19 & 0.35 & $\{0, 1, 2, 3, 89\}$ \\
    46 & 10.1 & \textbf{2} & 23 & 0.20 & $\{0, 46\}$ & 96 & 14.4 & \textbf{12} & 8 & 0.84 & $\{0, 1, 2, 3, 4, 5, 6, 7, 8, 9, 10, 41\}$ \\
    {\color{accentred}47} & 10.2 & \textbf{1} & 47 & 0.10 & $\{47\}$ & {\color{accentred}97} & 14.4 & \textbf{1} & 97 & 0.07 & $\{97\}$ \\
    48 & 10.3 & \textbf{8} & 6 & 0.78 & $\{0, 1, 2, 3, 4, 5, 6, 27\}$ & {\color{applegreen!70!black}98} & 14.5 & \textbf{14} & 7 & 0.96 & $\{0, 1, 2, 3, 4, 5, 6, 7, 8, 9, 10, 11, 12, 20\}$ \\
    49 & 10.4 & \textbf{7} & 7 & 0.67 & $\{0, 1, 2, 3, 4, 5, 34\}$ & 99 & 14.6 & \textbf{11} & 9 & 0.75 & $\{0, 1, 2, 3, 4, 5, 6, 7, 8, 9, 54\}$ \\
    {\color{applegreen!70!black}50} & 10.5 & \textbf{10} & 5 & 0.95 & $\{0, 1, 2, 3, 4, 5, 6, 7, 8, 14\}$ & 100 & 14.7 & \textbf{10} & 10 & 0.68 & $\{0, 1, 2, 3, 4, 5, 6, 7, 8, 64\}$ \\
\bottomrule
\end{tabular}
}
\end{table*}

\section{Concluding Remarks}

What begins as a simple puzzle about apples and pears reveals, upon closer inspection, a surprisingly rich interplay between two branches of elementary mathematics: combinatorics and number theory.  The connection to triangular numbers (Section~\ref{sec:perfect}), the sharp collapse at primes (Section~\ref{sec:primes}), and the role of highly composite numbers all point to a structure that is deeper than the puzzle's recreational origins might suggest.

The heart of the problem lies in recognising that it decomposes into two independent constraints.  The apples impose a divisibility condition, namely that the number of baskets must divide~$N$, while the pears impose a combinatorial one: $n$~distinct non-negative integers must sum to exactly~$N$, which by the minimum-sum lemma (Lemma~\ref{lem:minsum}) requires $n(n-1)/2 \le N$.  Neither constraint alone is difficult; it is their intersection that gives the problem its character.

For the original puzzle with $N = 60$, the answer is $10$~baskets, each holding $6$~apples, with pears distributed as $\{0, 1, 2, \ldots, 8, 24\}$ (see Figure~\ref{fig:closing}).  The general solution (Theorem~\ref{thm:general}), that $n_{\max}$ is the largest divisor of~$N$ not exceeding $(1 + \sqrt{1+8N})/2$, is elegant in statement but resists a tidy closed form.  This is because the divisibility constraint ties the answer to the prime factorisation of~$N$ in ways that no smooth function can capture~\cite{hardy2008}.

This dependence on arithmetic structure leads to the natural classification developed in Section~\ref{sec:perfect}.  At one extreme lie the \emph{perfect} values, $N = n(n-1)/2$ for odd~$n$, where both constraints are simultaneously tight and the pear distribution is uniquely $\{0,1,\ldots,n-1\}$.  At the other extreme sit the primes: for any prime $p \ge 17$, the only available divisors are~$1$ and $p$ itself, and since $p$ far exceeds the pear bound, the problem collapses to a single basket, which is the worst possible outcome.  Between these poles, highly composite numbers in the sense of Ramanujan~\cite{ramanujan1915} perform reliably well, their dense divisor sets ensuring that one lands close to the theoretical maximum.

The asymptotic growth of $n_{\max}$ as $\sqrt{2N}$ is governed by the pear constraint, but whether a given~$N$ actually achieves this rate depends entirely on whether it has a divisor in the right neighbourhood.  This tension between a smooth analytic bound and an irregular arithmetic condition is what makes the problem, in the author's view, a particularly appealing example of how elementary questions can open doors to deeper structure.  The number-theoretic landscape that emerges is, in the spirit of Hardy and Wright~\cite{hardy2008}, one defined by the interplay of regularity and irregularity in the distribution of divisors.

\end{document}

%% file: closing_baskets.tex
  \draw[basketbrown, thick, fill=basketbrown!5]
    (-0.55,-0.08) .. controls (-0.65,0.26) and (-0.4,0.5525) ..
    (0.0,0.65) .. controls (0.4,0.5525) and (0.65,0.26) ..
    (0.55,-0.08) -- cycle;
  \fill[red!70!black] (-0.200,0.080) circle (0.06);
  \draw[brown!50!black, line width=0.2pt] (-0.200,0.140) -- (-0.195,0.170);
  \fill[red!70!black] (0.000,0.080) circle (0.06);
  \draw[brown!50!black, line width=0.2pt] (0.000,0.140) -- (0.005,0.170);
  \fill[red!70!black] (0.200,0.080) circle (0.06);
  \draw[brown!50!black, line width=0.2pt] (0.200,0.140) -- (0.205,0.170);
  \fill[red!70!black] (-0.200,0.260) circle (0.06);
  \draw[brown!50!black, line width=0.2pt] (-0.200,0.320) -- (-0.195,0.350);
  \fill[red!70!black] (0.000,0.260) circle (0.06);
  \draw[brown!50!black, line width=0.2pt] (0.000,0.320) -- (0.005,0.350);
  \fill[red!70!black] (0.200,0.260) circle (0.06);
  \draw[brown!50!black, line width=0.2pt] (0.200,0.320) -- (0.205,0.350);
  \node[font=\tiny, below] at (0.0,-0.22) {0};
  \draw[basketbrown, thick, fill=basketbrown!5]
    (1.0,-0.08) .. controls (0.9,0.268) and (1.15,0.5695) ..
    (1.55,0.67) .. controls (1.9500000000000002,0.5695) and (2.2,0.268) ..
    (2.1,-0.08) -- cycle;
  \fill[red!70!black] (1.350,0.080) circle (0.06);
  \draw[brown!50!black, line width=0.2pt] (1.350,0.140) -- (1.355,0.170);
  \fill[red!70!black] (1.550,0.080) circle (0.06);
  \draw[brown!50!black, line width=0.2pt] (1.550,0.140) -- (1.555,0.170);
  \fill[red!70!black] (1.750,0.080) circle (0.06);
  \draw[brown!50!black, line width=0.2pt] (1.750,0.140) -- (1.755,0.170);
  \fill[red!70!black] (1.350,0.260) circle (0.06);
  \draw[brown!50!black, line width=0.2pt] (1.350,0.320) -- (1.355,0.350);
  \fill[red!70!black] (1.550,0.260) circle (0.06);
  \draw[brown!50!black, line width=0.2pt] (1.550,0.320) -- (1.555,0.350);
  \fill[red!70!black] (1.750,0.260) circle (0.06);
  \draw[brown!50!black, line width=0.2pt] (1.750,0.320) -- (1.755,0.350);
  \fill[yellow!55!green!70] (1.550,0.400) ellipse (0.05 and 0.04);
  \fill[yellow!50!green!60] (1.550,0.435) ellipse (0.035 and 0.035);
  \node[font=\tiny, below] at (1.55,-0.22) {1};
  \draw[basketbrown, thick, fill=basketbrown!5]
    (2.55,-0.08) .. controls (2.45,0.268) and (2.7,0.5695) ..
    (3.1,0.67) .. controls (3.5,0.5695) and (3.75,0.268) ..
    (3.6500000000000004,-0.08) -- cycle;
  \fill[red!70!black] (2.900,0.080) circle (0.06);
  \draw[brown!50!black, line width=0.2pt] (2.900,0.140) -- (2.905,0.170);
  \fill[red!70!black] (3.100,0.080) circle (0.06);
  \draw[brown!50!black, line width=0.2pt] (3.100,0.140) -- (3.105,0.170);
  \fill[red!70!black] (3.300,0.080) circle (0.06);
  \draw[brown!50!black, line width=0.2pt] (3.300,0.140) -- (3.305,0.170);
  \fill[red!70!black] (2.900,0.260) circle (0.06);
  \draw[brown!50!black, line width=0.2pt] (2.900,0.320) -- (2.905,0.350);
  \fill[red!70!black] (3.100,0.260) circle (0.06);
  \draw[brown!50!black, line width=0.2pt] (3.100,0.320) -- (3.105,0.350);
  \fill[red!70!black] (3.300,0.260) circle (0.06);
  \draw[brown!50!black, line width=0.2pt] (3.300,0.320) -- (3.305,0.350);
  \fill[yellow!55!green!70] (2.980,0.400) ellipse (0.05 and 0.04);
  \fill[yellow!50!green!60] (2.980,0.435) ellipse (0.035 and 0.035);
  \fill[yellow!55!green!70] (3.220,0.400) ellipse (0.05 and 0.04);
  \fill[yellow!50!green!60] (3.220,0.435) ellipse (0.035 and 0.035);
  \node[font=\tiny, below] at (3.1,-0.22) {2};
  \draw[basketbrown, thick, fill=basketbrown!5]
    (4.1000000000000005,-0.08) .. controls (4.0,0.268) and (4.25,0.5695) ..
    (4.65,0.67) .. controls (5.050000000000001,0.5695) and (5.300000000000001,0.268) ..
    (5.2,-0.08) -- cycle;
  \fill[red!70!black] (4.450,0.080) circle (0.06);
  \draw[brown!50!black, line width=0.2pt] (4.450,0.140) -- (4.455,0.170);
  \fill[red!70!black] (4.650,0.080) circle (0.06);
  \draw[brown!50!black, line width=0.2pt] (4.650,0.140) -- (4.655,0.170);
  \fill[red!70!black] (4.850,0.080) circle (0.06);
  \draw[brown!50!black, line width=0.2pt] (4.850,0.140) -- (4.855,0.170);
  \fill[red!70!black] (4.450,0.260) circle (0.06);
  \draw[brown!50!black, line width=0.2pt] (4.450,0.320) -- (4.455,0.350);
  \fill[red!70!black] (4.650,0.260) circle (0.06);
  \draw[brown!50!black, line width=0.2pt] (4.650,0.320) -- (4.655,0.350);
  \fill[red!70!black] (4.850,0.260) circle (0.06);
  \draw[brown!50!black, line width=0.2pt] (4.850,0.320) -- (4.855,0.350);
  \fill[yellow!55!green!70] (4.410,0.400) ellipse (0.05 and 0.04);
  \fill[yellow!50!green!60] (4.410,0.435) ellipse (0.035 and 0.035);
  \fill[yellow!55!green!70] (4.650,0.400) ellipse (0.05 and 0.04);
  \fill[yellow!50!green!60] (4.650,0.435) ellipse (0.035 and 0.035);
  \fill[yellow!55!green!70] (4.890,0.400) ellipse (0.05 and 0.04);
  \fill[yellow!50!green!60] (4.890,0.435) ellipse (0.035 and 0.035);
  \node[font=\tiny, below] at (4.65,-0.22) {3};
  \draw[basketbrown, thick, fill=basketbrown!5]
    (5.65,-0.08) .. controls (5.55,0.316) and (5.8,0.6714999999999999) ..
    (6.2,0.7899999999999999) .. controls (6.6000000000000005,0.6714999999999999) and (6.8500000000000005,0.316) ..
    (6.75,-0.08) -- cycle;
  \fill[red!70!black] (6.000,0.080) circle (0.06);
  \draw[brown!50!black, line width=0.2pt] (6.000,0.140) -- (6.005,0.170);
  \fill[red!70!black] (6.200,0.080) circle (0.06);
  \draw[brown!50!black, line width=0.2pt] (6.200,0.140) -- (6.205,0.170);
  \fill[red!70!black] (6.400,0.080) circle (0.06);
  \draw[brown!50!black, line width=0.2pt] (6.400,0.140) -- (6.405,0.170);
  \fill[red!70!black] (6.000,0.260) circle (0.06);
  \draw[brown!50!black, line width=0.2pt] (6.000,0.320) -- (6.005,0.350);
  \fill[red!70!black] (6.200,0.260) circle (0.06);
  \draw[brown!50!black, line width=0.2pt] (6.200,0.320) -- (6.205,0.350);
  \fill[red!70!black] (6.400,0.260) circle (0.06);
  \draw[brown!50!black, line width=0.2pt] (6.400,0.320) -- (6.405,0.350);
  \fill[yellow!55!green!70] (5.960,0.400) ellipse (0.05 and 0.04);
  \fill[yellow!50!green!60] (5.960,0.435) ellipse (0.035 and 0.035);
  \fill[yellow!55!green!70] (6.200,0.400) ellipse (0.05 and 0.04);
  \fill[yellow!50!green!60] (6.200,0.435) ellipse (0.035 and 0.035);
  \fill[yellow!55!green!70] (6.440,0.400) ellipse (0.05 and 0.04);
  \fill[yellow!50!green!60] (6.440,0.435) ellipse (0.035 and 0.035);
  \fill[yellow!55!green!70] (6.200,0.520) ellipse (0.05 and 0.04);
  \fill[yellow!50!green!60] (6.200,0.555) ellipse (0.035 and 0.035);
  \node[font=\tiny, below] at (6.2,-0.22) {4};
  \draw[basketbrown, thick, fill=basketbrown!5]
    (7.2,-0.08) .. controls (7.1,0.316) and (7.35,0.6714999999999999) ..
    (7.75,0.7899999999999999) .. controls (8.15,0.6714999999999999) and (8.4,0.316) ..
    (8.3,-0.08) -- cycle;
  \fill[red!70!black] (7.550,0.080) circle (0.06);
  \draw[brown!50!black, line width=0.2pt] (7.550,0.140) -- (7.555,0.170);
  \fill[red!70!black] (7.750,0.080) circle (0.06);
  \draw[brown!50!black, line width=0.2pt] (7.750,0.140) -- (7.755,0.170);
  \fill[red!70!black] (7.950,0.080) circle (0.06);
  \draw[brown!50!black, line width=0.2pt] (7.950,0.140) -- (7.955,0.170);
  \fill[red!70!black] (7.550,0.260) circle (0.06);
  \draw[brown!50!black, line width=0.2pt] (7.550,0.320) -- (7.555,0.350);
  \fill[red!70!black] (7.750,0.260) circle (0.06);
  \draw[brown!50!black, line width=0.2pt] (7.750,0.320) -- (7.755,0.350);
  \fill[red!70!black] (7.950,0.260) circle (0.06);
  \draw[brown!50!black, line width=0.2pt] (7.950,0.320) -- (7.955,0.350);
  \fill[yellow!55!green!70] (7.510,0.400) ellipse (0.05 and 0.04);
  \fill[yellow!50!green!60] (7.510,0.435) ellipse (0.035 and 0.035);
  \fill[yellow!55!green!70] (7.750,0.400) ellipse (0.05 and 0.04);
  \fill[yellow!50!green!60] (7.750,0.435) ellipse (0.035 and 0.035);
  \fill[yellow!55!green!70] (7.990,0.400) ellipse (0.05 and 0.04);
  \fill[yellow!50!green!60] (7.990,0.435) ellipse (0.035 and 0.035);
  \fill[yellow!55!green!70] (7.630,0.520) ellipse (0.05 and 0.04);
  \fill[yellow!50!green!60] (7.630,0.555) ellipse (0.035 and 0.035);
  \fill[yellow!55!green!70] (7.870,0.520) ellipse (0.05 and 0.04);
  \fill[yellow!50!green!60] (7.870,0.555) ellipse (0.035 and 0.035);
  \node[font=\tiny, below] at (7.75,-0.22) {5};
  \draw[basketbrown, thick, fill=basketbrown!5]
    (8.75,-0.08) .. controls (8.65,0.316) and (8.9,0.6714999999999999) ..
    (9.3,0.7899999999999999) .. controls (9.700000000000001,0.6714999999999999) and (9.950000000000001,0.316) ..
    (9.850000000000001,-0.08) -- cycle;
  \fill[red!70!black] (9.100,0.080) circle (0.06);
  \draw[brown!50!black, line width=0.2pt] (9.100,0.140) -- (9.105,0.170);
  \fill[red!70!black] (9.300,0.080) circle (0.06);
  \draw[brown!50!black, line width=0.2pt] (9.300,0.140) -- (9.305,0.170);
  \fill[red!70!black] (9.500,0.080) circle (0.06);
  \draw[brown!50!black, line width=0.2pt] (9.500,0.140) -- (9.505,0.170);
  \fill[red!70!black] (9.100,0.260) circle (0.06);
  \draw[brown!50!black, line width=0.2pt] (9.100,0.320) -- (9.105,0.350);
  \fill[red!70!black] (9.300,0.260) circle (0.06);
  \draw[brown!50!black, line width=0.2pt] (9.300,0.320) -- (9.305,0.350);
  \fill[red!70!black] (9.500,0.260) circle (0.06);
  \draw[brown!50!black, line width=0.2pt] (9.500,0.320) -- (9.505,0.350);
  \fill[yellow!55!green!70] (9.060,0.400) ellipse (0.05 and 0.04);
  \fill[yellow!50!green!60] (9.060,0.435) ellipse (0.035 and 0.035);
  \fill[yellow!55!green!70] (9.300,0.400) ellipse (0.05 and 0.04);
  \fill[yellow!50!green!60] (9.300,0.435) ellipse (0.035 and 0.035);
  \fill[yellow!55!green!70] (9.540,0.400) ellipse (0.05 and 0.04);
  \fill[yellow!50!green!60] (9.540,0.435) ellipse (0.035 and 0.035);
  \fill[yellow!55!green!70] (9.060,0.520) ellipse (0.05 and 0.04);
  \fill[yellow!50!green!60] (9.060,0.555) ellipse (0.035 and 0.035);
  \fill[yellow!55!green!70] (9.300,0.520) ellipse (0.05 and 0.04);
  \fill[yellow!50!green!60] (9.300,0.555) ellipse (0.035 and 0.035);
  \fill[yellow!55!green!70] (9.540,0.520) ellipse (0.05 and 0.04);
  \fill[yellow!50!green!60] (9.540,0.555) ellipse (0.035 and 0.035);
  \node[font=\tiny, below] at (9.3,-0.22) {6};
  \draw[basketbrown, thick, fill=basketbrown!5]
    (10.299999999999999,-0.08) .. controls (10.2,0.36400000000000005) and (10.45,0.7735) ..
    (10.85,0.91) .. controls (11.25,0.7735) and (11.5,0.36400000000000005) ..
    (11.4,-0.08) -- cycle;
  \fill[red!70!black] (10.650,0.080) circle (0.06);
  \draw[brown!50!black, line width=0.2pt] (10.650,0.140) -- (10.655,0.170);
  \fill[red!70!black] (10.850,0.080) circle (0.06);
  \draw[brown!50!black, line width=0.2pt] (10.850,0.140) -- (10.855,0.170);
  \fill[red!70!black] (11.050,0.080) circle (0.06);
  \draw[brown!50!black, line width=0.2pt] (11.050,0.140) -- (11.055,0.170);
  \fill[red!70!black] (10.650,0.260) circle (0.06);
  \draw[brown!50!black, line width=0.2pt] (10.650,0.320) -- (10.655,0.350);
  \fill[red!70!black] (10.850,0.260) circle (0.06);
  \draw[brown!50!black, line width=0.2pt] (10.850,0.320) -- (10.855,0.350);
  \fill[red!70!black] (11.050,0.260) circle (0.06);
  \draw[brown!50!black, line width=0.2pt] (11.050,0.320) -- (11.055,0.350);
  \fill[yellow!55!green!70] (10.610,0.400) ellipse (0.05 and 0.04);
  \fill[yellow!50!green!60] (10.610,0.435) ellipse (0.035 and 0.035);
  \fill[yellow!55!green!70] (10.850,0.400) ellipse (0.05 and 0.04);
  \fill[yellow!50!green!60] (10.850,0.435) ellipse (0.035 and 0.035);
  \fill[yellow!55!green!70] (11.090,0.400) ellipse (0.05 and 0.04);
  \fill[yellow!50!green!60] (11.090,0.435) ellipse (0.035 and 0.035);
  \fill[yellow!55!green!70] (10.610,0.520) ellipse (0.05 and 0.04);
  \fill[yellow!50!green!60] (10.610,0.555) ellipse (0.035 and 0.035);
  \fill[yellow!55!green!70] (10.850,0.520) ellipse (0.05 and 0.04);
  \fill[yellow!50!green!60] (10.850,0.555) ellipse (0.035 and 0.035);
  \fill[yellow!55!green!70] (11.090,0.520) ellipse (0.05 and 0.04);
  \fill[yellow!50!green!60] (11.090,0.555) ellipse (0.035 and 0.035);
  \fill[yellow!55!green!70] (10.850,0.640) ellipse (0.05 and 0.04);
  \fill[yellow!50!green!60] (10.850,0.675) ellipse (0.035 and 0.035);
  \node[font=\tiny, below] at (10.85,-0.22) {7};
  \draw[basketbrown, thick, fill=basketbrown!5]
    (11.85,-0.08) .. controls (11.75,0.36400000000000005) and (12.0,0.7735) ..
    (12.4,0.91) .. controls (12.8,0.7735) and (13.05,0.36400000000000005) ..
    (12.950000000000001,-0.08) -- cycle;
  \fill[red!70!black] (12.200,0.080) circle (0.06);
  \draw[brown!50!black, line width=0.2pt] (12.200,0.140) -- (12.205,0.170);
  \fill[red!70!black] (12.400,0.080) circle (0.06);
  \draw[brown!50!black, line width=0.2pt] (12.400,0.140) -- (12.405,0.170);
  \fill[red!70!black] (12.600,0.080) circle (0.06);
  \draw[brown!50!black, line width=0.2pt] (12.600,0.140) -- (12.605,0.170);
  \fill[red!70!black] (12.200,0.260) circle (0.06);
  \draw[brown!50!black, line width=0.2pt] (12.200,0.320) -- (12.205,0.350);
  \fill[red!70!black] (12.400,0.260) circle (0.06);
  \draw[brown!50!black, line width=0.2pt] (12.400,0.320) -- (12.405,0.350);
  \fill[red!70!black] (12.600,0.260) circle (0.06);
  \draw[brown!50!black, line width=0.2pt] (12.600,0.320) -- (12.605,0.350);
  \fill[yellow!55!green!70] (12.160,0.400) ellipse (0.05 and 0.04);
  \fill[yellow!50!green!60] (12.160,0.435) ellipse (0.035 and 0.035);
  \fill[yellow!55!green!70] (12.400,0.400) ellipse (0.05 and 0.04);
  \fill[yellow!50!green!60] (12.400,0.435) ellipse (0.035 and 0.035);
  \fill[yellow!55!green!70] (12.640,0.400) ellipse (0.05 and 0.04);
  \fill[yellow!50!green!60] (12.640,0.435) ellipse (0.035 and 0.035);
  \fill[yellow!55!green!70] (12.160,0.520) ellipse (0.05 and 0.04);
  \fill[yellow!50!green!60] (12.160,0.555) ellipse (0.035 and 0.035);
  \fill[yellow!55!green!70] (12.400,0.520) ellipse (0.05 and 0.04);
  \fill[yellow!50!green!60] (12.400,0.555) ellipse (0.035 and 0.035);
  \fill[yellow!55!green!70] (12.640,0.520) ellipse (0.05 and 0.04);
  \fill[yellow!50!green!60] (12.640,0.555) ellipse (0.035 and 0.035);
  \fill[yellow!55!green!70] (12.280,0.640) ellipse (0.05 and 0.04);
  \fill[yellow!50!green!60] (12.280,0.675) ellipse (0.035 and 0.035);
  \fill[yellow!55!green!70] (12.520,0.640) ellipse (0.05 and 0.04);
  \fill[yellow!50!green!60] (12.520,0.675) ellipse (0.035 and 0.035);
  \node[font=\tiny, below] at (12.4,-0.22) {8};
  \draw[basketbrown, thick, fill=basketbrown!5]
    (13.4,-0.08) .. controls (13.3,0.6040000000000001) and (13.55,1.2834999999999999) ..
    (13.950000000000001,1.51) .. controls (14.350000000000001,1.2834999999999999) and (14.600000000000001,0.6040000000000001) ..
    (14.500000000000002,-0.08) -- cycle;
  \fill[red!70!black] (13.750,0.080) circle (0.06);
  \draw[brown!50!black, line width=0.2pt] (13.750,0.140) -- (13.755,0.170);
  \fill[red!70!black] (13.950,0.080) circle (0.06);
  \draw[brown!50!black, line width=0.2pt] (13.950,0.140) -- (13.955,0.170);
  \fill[red!70!black] (14.150,0.080) circle (0.06);
  \draw[brown!50!black, line width=0.2pt] (14.150,0.140) -- (14.155,0.170);
  \fill[red!70!black] (13.750,0.260) circle (0.06);
  \draw[brown!50!black, line width=0.2pt] (13.750,0.320) -- (13.755,0.350);
  \fill[red!70!black] (13.950,0.260) circle (0.06);
  \draw[brown!50!black, line width=0.2pt] (13.950,0.320) -- (13.955,0.350);
  \fill[red!70!black] (14.150,0.260) circle (0.06);
  \draw[brown!50!black, line width=0.2pt] (14.150,0.320) -- (14.155,0.350);
  \fill[yellow!55!green!70] (13.710,0.400) ellipse (0.05 and 0.04);
  \fill[yellow!50!green!60] (13.710,0.435) ellipse (0.035 and 0.035);
  \fill[yellow!55!green!70] (13.950,0.400) ellipse (0.05 and 0.04);
  \fill[yellow!50!green!60] (13.950,0.435) ellipse (0.035 and 0.035);
  \fill[yellow!55!green!70] (14.190,0.400) ellipse (0.05 and 0.04);
  \fill[yellow!50!green!60] (14.190,0.435) ellipse (0.035 and 0.035);
  \fill[yellow!55!green!70] (13.710,0.520) ellipse (0.05 and 0.04);
  \fill[yellow!50!green!60] (13.710,0.555) ellipse (0.035 and 0.035);
  \fill[yellow!55!green!70] (13.950,0.520) ellipse (0.05 and 0.04);
  \fill[yellow!50!green!60] (13.950,0.555) ellipse (0.035 and 0.035);
  \fill[yellow!55!green!70] (14.190,0.520) ellipse (0.05 and 0.04);
  \fill[yellow!50!green!60] (14.190,0.555) ellipse (0.035 and 0.035);
  \fill[yellow!55!green!70] (13.710,0.640) ellipse (0.05 and 0.04);
  \fill[yellow!50!green!60] (13.710,0.675) ellipse (0.035 and 0.035);
  \fill[yellow!55!green!70] (13.950,0.640) ellipse (0.05 and 0.04);
  \fill[yellow!50!green!60] (13.950,0.675) ellipse (0.035 and 0.035);
  \fill[yellow!55!green!70] (14.190,0.640) ellipse (0.05 and 0.04);
  \fill[yellow!50!green!60] (14.190,0.675) ellipse (0.035 and 0.035);
  \fill[yellow!55!green!70] (13.710,0.760) ellipse (0.05 and 0.04);
  \fill[yellow!50!green!60] (13.710,0.795) ellipse (0.035 and 0.035);
  \fill[yellow!55!green!70] (13.950,0.760) ellipse (0.05 and 0.04);
  \fill[yellow!50!green!60] (13.950,0.795) ellipse (0.035 and 0.035);
  \fill[yellow!55!green!70] (14.190,0.760) ellipse (0.05 and 0.04);
  \fill[yellow!50!green!60] (14.190,0.795) ellipse (0.035 and 0.035);
  \fill[yellow!55!green!70] (13.710,0.880) ellipse (0.05 and 0.04);
  \fill[yellow!50!green!60] (13.710,0.915) ellipse (0.035 and 0.035);
  \fill[yellow!55!green!70] (13.950,0.880) ellipse (0.05 and 0.04);
  \fill[yellow!50!green!60] (13.950,0.915) ellipse (0.035 and 0.035);
  \fill[yellow!55!green!70] (14.190,0.880) ellipse (0.05 and 0.04);
  \fill[yellow!50!green!60] (14.190,0.915) ellipse (0.035 and 0.035);
  \fill[yellow!55!green!70] (13.710,1.000) ellipse (0.05 and 0.04);
  \fill[yellow!50!green!60] (13.710,1.035) ellipse (0.035 and 0.035);
  \fill[yellow!55!green!70] (13.950,1.000) ellipse (0.05 and 0.04);
  \fill[yellow!50!green!60] (13.950,1.035) ellipse (0.035 and 0.035);
  \fill[yellow!55!green!70] (14.190,1.000) ellipse (0.05 and 0.04);
  \fill[yellow!50!green!60] (14.190,1.035) ellipse (0.035 and 0.035);
  \fill[yellow!55!green!70] (13.710,1.120) ellipse (0.05 and 0.04);
  \fill[yellow!50!green!60] (13.710,1.155) ellipse (0.035 and 0.035);
  \fill[yellow!55!green!70] (13.950,1.120) ellipse (0.05 and 0.04);
  \fill[yellow!50!green!60] (13.950,1.155) ellipse (0.035 and 0.035);
  \fill[yellow!55!green!70] (14.190,1.120) ellipse (0.05 and 0.04);
  \fill[yellow!50!green!60] (14.190,1.155) ellipse (0.035 and 0.035);
  \fill[yellow!55!green!70] (13.710,1.240) ellipse (0.05 and 0.04);
  \fill[yellow!50!green!60] (13.710,1.275) ellipse (0.035 and 0.035);
  \fill[yellow!55!green!70] (13.950,1.240) ellipse (0.05 and 0.04);
  \fill[yellow!50!green!60] (13.950,1.275) ellipse (0.035 and 0.035);
  \fill[yellow!55!green!70] (14.190,1.240) ellipse (0.05 and 0.04);
  \fill[yellow!50!green!60] (14.190,1.275) ellipse (0.035 and 0.035);
  \node[font=\tiny, below] at (13.950000000000001,-0.22) {24};

%% file: basket_problem_nature.bbl
\begin{thebibliography}{10}

\bibitem{hardy2008}
G.\,H.~Hardy and E.\,M.~Wright,
\textit{An Introduction to the Theory of Numbers},
6th ed., Oxford University Press, 2008.
Chapters 16--18 cover the divisor function $d(n)$, its average order, and the distribution of divisors.

\bibitem{andrews1998}
G.\,E.~Andrews,
\textit{The Theory of Partitions},
Cambridge University Press, 1998.
The minimum-sum lemma (Lemma~\ref{lem:minsum}) is a special case of results on partitions into distinct parts; see Chapter~1.

\bibitem{ramanujan1915}
S.~Ramanujan,
``Highly composite numbers,''
\textit{Proceedings of the London Mathematical Society}, vol.~14, pp.~347--409, 1915.
The classification of integers by divisor richness in Section~4 echoes Ramanujan's study of numbers with unusually many divisors.

\bibitem{apostol1976}
T.\,M.~Apostol,
\textit{Introduction to Analytic Number Theory},
Springer, 1976.
Chapters 3--4 provide a rigorous treatment of the prime counting function $\pi(x)$ and the prime number theorem referenced in Section~\ref{sec:primes}.

\bibitem{hadamard1896}
J.~Hadamard,
``Sur la distribution des z\'eros de la fonction $\zeta(s)$ et ses cons\'equences arithm\'etiques,''
\textit{Bulletin de la Soci\'et\'e Math\'ematique de France}, vol.~24, pp.~199--220, 1896.

\bibitem{vallee1896}
C.-J.~de la Vall\'ee Poussin,
``Recherches analytiques sur la th\'eorie des nombres premiers,''
\textit{Annales de la Soci\'et\'e Scientifique de Bruxelles}, vol.~20, pp.~183--256, 1896.
References~\cite{hadamard1896} and~\cite{vallee1896} independently proved the prime number theorem, $\pi(x) \sim x/\ln x$.

\bibitem{conway1996}
J.\,H.~Conway and R.\,K.~Guy,
\textit{The Book of Numbers},
Copernicus (Springer), 1996.
A survey of number-theoretic curiosities, including triangular numbers and related combinatorial puzzles.

\bibitem{gardner1997}
M.~Gardner,
\textit{The Last Recreations: Hydras, Eggs, and Other Mathematical Mystifications},
Copernicus (Springer), 1997.
A classic collection of recreational mathematics problems in the spirit of the present puzzle.

\bibitem{oeis_triangular}
OEIS Foundation,
``Sequence A000217: Triangular numbers,''
\textit{The On-Line Encyclopedia of Integer Sequences}, 2024.
Available at \url{https://oeis.org/A000217}.

\bibitem{niven1991}
I.~Niven, H.\,S.~Zuckerman, and H.\,L.~Montgomery,
\textit{An Introduction to the Theory of Numbers},
5th ed., Wiley, 1991.
An alternative reference for the divisibility and prime-theoretic results used in Theorem~\ref{thm:general}.

\end{thebibliography}
